\documentclass[reqno,11pt]{amsart}
\usepackage{amsmath, amsthm, amssymb, geometry, enumerate, latexsym}
\usepackage[all]{xy}
\def\ro{\rho}
\def\si{\sigma}

\def\vph{\varphi}
\def\ch{\chi}

\def\om{\omega}

\def\cA{{\mathcal A}}
\def\cB{{\mathcal B}}

\def\cF{{\mathcal F}}

\def\cK{{\mathcal K}}

\def\cM{{\mathcal M}}
\def\cN{{\mathcal N}}
\def\cO{{\mathcal O}}

\def\cV{{\mathcal V}}

\def\cX{{\mathcal X}}

\def\ang#1{{\langle #1 \rangle}}

\def\P{{\mathbb P}}

\def\Z{{\mathbb Z}}
\def\N{{\mathbb N}}
\def\m{{\mathfrak m}}

\def\add{\operatorname{add}}

\def\Aut{\operatorname{Aut}}

\def\fchar{\operatorname{char}}

\def\cd{\operatorname{cd}}
\def\Coker{\operatorname{Coker}}
\def\depth{\operatorname{depth}}

\def\dim{\operatorname{dim}}

\def\Ext{\operatorname{Ext}}

\def\gldim{\operatorname{gldim}}

\def\GrAut{\operatorname{GrAut}}

\def\Hom{\operatorname{Hom}}
\def\hdet{\operatorname{hdet}}

\def\injdim{\operatorname{injdim}}

\def\ldim{\operatorname{ldim}}

\def\op{\operatorname{op}}
\def\projdim{\operatorname{projdim}}

\def\Proj{\operatorname{Proj}}

\def\sup{\operatorname{sup}}

\newcommand{\uHom}{\operatorname{\underline{Hom}}\nolimits}
\newcommand{\uEnd}{\operatorname{\underline{End}}\nolimits}

\newcommand{\uExt}{\operatorname{\underline{Ext}}\nolimits}
\newcommand{\uH}{\operatorname{\underline{H}}\nolimits}
\def\RuHom{\operatorname{R\underline{Hom}}}
\def\G{\Gamma}
\def\RuG{\operatorname{R\underline {\G }}}
\newcommand{\uG}{\operatorname{\underline{\Gamma}}\nolimits}
\newcommand{\Lotimes}{\otimes^{\operatorname L}}

\def\add{\mathsf{add}\,}
\def\grmod{\mathsf{grmod}\,}
\def\GrMod{\mathsf{GrMod}\,}
\def\CM#1{\mathsf{CM^{gr}}(#1)}
\def\tails{\mathsf{tails}\,}
\def\Tails{\mathsf{Tails}\,}

\def\coh{\mathsf{coh}\,}

\def\uCM#1{\mathsf{\underline{CM}^{gr}}(#1)}
\def\tors{\mathsf{tors}\,}
\def\Tors{\mathsf{Tors}\,}
\def\D#1{\mathsf{D(}#1\mathsf{)}}
\def\Db#1{\mathsf{D^b(}#1\mathsf{)}}

\def\mod{\mathsf{mod}\,}

\def\rnum#1{\expandafter{\romannumeral #1}}
\def\Rnum#1{\uppercase\expandafter{\romannumeral #1}}

\theoremstyle{plain} 
\newtheorem{thm}{Theorem}[section]
\newtheorem{cor}[thm]{Corollary}
\newtheorem{lem}[thm]{Lemma}
\newtheorem{prop}[thm]{Proposition}

\theoremstyle{definition}
\newtheorem{dfn}[thm]{Definition}
\newtheorem{ex}[thm]{Example}
\newtheorem{ques}[thm]{Question}

\theoremstyle{remark}

\newtheorem{rem}[thm]{Remark}

\newcommand{\thmref}[1]{Theorem~\ref{#1}}
\newcommand{\lemref}[1]{Lemma~\ref{#1}}
\newcommand{\corref}[1]{Corollary~\ref{#1}}
\newcommand{\propref}[1]{Proposition~\ref{#1}}
\newcommand{\exref}[1]{Example~\ref{#1}}

\newcommand{\remref}[1]{Remark~\ref{#1}}
\newcommand{\quesref}[1]{Question~\ref{#1}}

\numberwithin{equation}{section}

\def\da{\dagger}
\begin{document}

\title
[Cluster tilting modules and noncommutative projective schemes]
{Cluster tilting modules and noncommutative projective schemes}
\author{Kenta Ueyama}
\address{
Department of Mathematics, 
Faculty of Education,
Hirosaki University, 
1 Bunkyocho, Hirosaki, Aomori 036-8560, Japan}
\email{k-ueyama@hirosaki-u.ac.jp} 
\thanks{The author was supported by JSPS Grant-in-Aid for Young Scientists (B) 15K17503.}


\begin{abstract}
In this paper, we study the relationship between equivalences of noncommutative projective schemes
and cluster tilting modules.
In particular, we prove the following result.
Let $A$ be an AS-Gorenstein algebra of dimension $d\geq 2$
and ${\mathsf{tails}\,} A$ the noncommutative projective scheme associated to $A$.
If $\operatorname{gldim}({\mathsf{tails}\,} A)< \infty$ and
$A$ has a $(d-1)$-cluster tilting module $X$ satisfying that its graded endomorphism algebra is $\mathbb N$-graded,
then the graded endomorphism algebra $B$ of a basic $(d-1)$-cluster tilting submodule of $X$
is a two-sided noetherian $\mathbb N$-graded AS-regular algebra over $B_0$ of global dimension $d$
such that ${\mathsf{tails}\,} B$ is equivalent to ${\mathsf{tails}\,} A$.
\end{abstract}

\maketitle

\section{Introduction}

In \cite{AZ}, Artin and Zhang introduced the notion of a noncommutative projective scheme, and 
established a fundamental and comprehensive theory of noncommutative projective schemes.
Since the study of the categories of coherent sheaves on commutative projective schemes
(or their derived categories) is of increasing importance in algebraic geometry,
the study of noncommutative projective schemes has been one of the most major projects
in noncommutative projective geometry.

Let $A, A'$ be right noetherian graded algebras,
and $\tails A, \tails A'$ the noncommutative projective schemes associated to $A$ and $A'$ respectively.
Clearly, if $A \cong A'$ as graded algebras, then $\tails A \cong \tails A'$.
It is well-known that the converse does not hold, so the following question is natural to ask. %and important

\begin{ques} \label{ques}
Given a right noetherian graded algebra $A$, can we find a better homogeneous coordinate ring of $\tails A$?
That is, can we find a better graded algebra $B$ (e.g. $\gldim B < \infty$) such that
$\tails B \cong \tails A$?
\end{ques}

For example, if we consider the commutative graded algebra $A = k[x,y,z,w]/(xw-yz)$,
then $\tails A \ (\cong \coh \P^1 \times \P^1)$ is not equivalent to $\tails k[x_1,\dots,x_n] \ (\cong \coh \P^{n-1})$,
but we can find the noncommutative graded algebra $B = k\ang{x,y}/(x^2y-yx^2, y^2x-xy^2)$ of global dimension $3$ such that
$\tails A \cong \tails B$, so the above question has a significant meaning in noncommutative projective geometry.

The purpose of this paper is to give an answer to \quesref{ques} by investigating cluster tilting modules.
Cluster tilting modules are crucial in the study of higher-dimensional analogues of Auslander-Reiten theory,
and also attract attention from the viewpoint of Van den Bergh's noncommutative crepant resolutions.
In particular, cluster tilting modules have been extensively studied for a certain class of algebras, called orders,
including commutative Cohen-Macaulay rings and finite dimensional algebras (see \cite{I}).
One of the motivations of this paper is to develop cluster tilting theory for non-orders
in terms of noncommutative projective geometry.

The main result of this paper is as follows.
Let $A$ be a two-sided noetherian connected graded algebra satisfying $\ch$
and let $X$ be a finitely generated graded right $A$-module.
If $A$ is an AS-Gorenstein algebra of dimension $d\geq 2$
and $X$ is a $(d-1)$-cluster tilting module satisfying some additional conditions,
then the graded endomorphism algebra $B = \uEnd_A(X)$ is a two-sided noetherian ASF-regular algebra of global dimension $d$
such that the functors
\begin{align*}
&\tails B \to \tails A \;\text{induced by}\; -\otimes_B X, \;\text{and}\\
&\tails B^{\op} \to \tails A^{\op} \;\text{induced by}\; \uHom_A(X,A) \otimes_B -
\end{align*}
are equivalences (\thmref{thm:main} (1)).
Moreover a certain converse statement also holds (\thmref{thm:main} (2)).
As a corollary of this result, we can give an answer to \quesref{ques} as follows.

\begin{thm} \textnormal{(\corref{cor:main})} \label{thm1}
Let $A$ be an AS-Gorenstein algebra of dimension $d\geq 2$.
If $\gldim(\tails A)< \infty$ and $A$ has a $(d-1)$-cluster tilting module $X$ satisfying that its graded endomorphism algebra is $\N$-graded,
then the graded endomorphism algebra $B$ of a basic $(d-1)$-cluster tilting submodule of $X$
is a two-sided noetherian $\N$-graded AS-regular algebra over $B_0$ of global dimension $d$
such that $\tails B \cong \tails A$.
\end{thm}

We note that the notions of ASF-regular and AS-regular over $R$ were recently introduced by Minamoto and Mori \cite{MM},
and these are natural generalizations of AS-regular algebras for $\N$-graded (not necessarily connected graded) algebras.
%% (see \cite{MM})
A comparison theorem for these algebras is described in \thmref{thm:com} (see also \corref{cor:com}).

\section{Preliminaries}

Throughout, let $k$ be a field.
A graded $k$-vector space $V=\bigoplus_{i\in\Z}V_i$ is called locally finite if $\dim_k V_i< \infty$ for all $i \in \Z$,
and it is called left (resp. right) bounded if $V_i =0$ for all $i \ll 0$ (resp. $i \gg 0$).
We denote by $DV=\uHom_k(V,k)$ the graded vector space dual of a locally finite graded $k$-vector space $V$.

In this paper, a graded algebra means a $\Z$-graded algebra over $k$ unless otherwise stated.
For a graded algebra $A$, we denote by $\GrMod A$ the category of graded right $A$-modules with $A$-module homomorphisms of degree $0$,
and by $\grmod A$ the full subcategory consisting of finitely generated graded $A$-modules.
Note that if $A$ is right noetherian, then $\grmod A$ is an abelian category.
We denote by $A^{\op}$ the opposite algebra of $A$,
and by $A^e = A^{\op} \otimes_k A$ the enveloping algebra.
The category of graded left $A$-modules is identified with $\GrMod A^{\op}$, and
the category of graded $A$-$A$ bimodules is identified with $\GrMod A^e$.

For a graded module $M \in \GrMod A$ and an integer $n\in \Z$,
we define the shift $M(n)\in \GrMod A$ by $M(n)_i:=M_{n+i}$ for $i\in \Z$.
Note that the rule $M\mapsto M(n)$ is a $k$-linear autoequivalence for $\GrMod A$ and $\grmod A$, called the shift functor.
For $M, N\in \GrMod A$, we write $\Ext^i_{\GrMod A}(M, N)$ for the extension group in $\GrMod A$, and define
$$\uExt^i_A(M, N):=\bigoplus _{i\in \Z}\Ext^i_{\GrMod A}(M, N(i)).$$

Let $A$ be a right noetherian locally finite $\N$-graded algebra.
For $M \in \GrMod A$ and an integer $n\in \Z$,
we define the truncated submodule $M_{\geq n} \in \GrMod A$ by $M_{\geq n} :=\bigoplus _{i\geq n}M_i$.
We say that an element $x$ of a graded module $M \in \GrMod A$ is torsion if there exists a positive integer $n$ such that $xA_{\geq n} = 0$.
We denote by $t(M)$ the submodule of $M$ consisting of all torsion elements.
A graded module $M \in \GrMod A$ is called torsion if $M=t(M)$, and torsion-free if $t(M)=0$.
We denote by $\Tors A$ (resp. $\tors A$) the full subcategory of $\GrMod A$ (resp. $\grmod A$) consisting of torsion modules.
One can define the Serre quotient categories
\[ \Tails A = \GrMod A / \Tors A \quad \text{and} \quad  \tails A = \grmod A / \tors A. \]
Note that $\tails A$ is the full subcategory of noetherian objects of $\Tails A$.
The quotient functor is denoted by $\pi : \GrMod A \to \Tails A$.
%%and its right adjoint is denoted by $\om : \Tails A \to \GrMod A$ so that $\pi\om \cong \Id_{\Tails A}$
We often denote by $\cM = \pi M \in \Tails A$ the image of $M \in \GrMod A$.
Note that the shift functor preserves torsion modules,
so it induces a $k$-linear autoequivalence $\cM \mapsto \cM(n)$ for $\Tails A$ and $\tails A$, again called the shift functor.
For $\cM, \cN\in \Tails A$, we write $\Ext^i_{\Tails A}(\cM, \cN)$ for the extension group in $\Tails A$, and define
$$\uExt^i_\cA(\cM, \cN):=\bigoplus _{i\in \Z}\Ext^i_{\Tails A}(\cM, \cN(i)).$$
See \cite[Section 7]{AZ} for details on Ext-groups in $\tails A$.
Following Serre's theorem and the Gabriel-Rosenberg reconstruction theorem, $\tails A$ is called the noncommutative projective scheme associated to $A$ 
(see \cite {AZ} for details).
We define the global dimension of $\tails A$ by
$$ \gldim (\tails A) = \sup\{i \mid \Ext^i_{\tails A}(\cM, \cN) \neq0\; \text{for some}\; \cM, \cN \in \tails A \}.$$
If $\gldim A<\infty$, then it is clear  that $\gldim(\tails A)<\infty$.
The condition $\gldim (\tails A)<\infty$ is considered as a noncommutative graded isolated singularity property (see \cite{U5}, \cite{U6}, \cite{MU1}, \cite{MU2}).

Recall that we say that $\chi_i$ holds for $A$ if $\uExt^j_A(A/A_{\geq 1}, M)$ is finite dimensional over $k$ for every $M \in \grmod A$ and every $j\leq i$,
and we say that $\chi$ holds for $A$ if $\chi_i$ holds for every $i$.
The condition $\chi_i$ plays an essential role in the study of the noncommutative projective schemes $\tails A$ (see \cite {AZ}, \cite{YZ} for details).

We call $({\mathsf C}, \cO, s)$ an algebraic triple
if it consists of a $k$-linear abelian category ${\mathsf C}$,
an object $\cO\in {\mathsf C}$,
and a $k$-linear autoequivalence $s\in \Aut _k{\mathsf C}$.

\begin{dfn} (\cite{AZ})
Let $({\mathsf C}, \cO, s)$ be an algebraic triple.  We say that the pair
$(\cO, s)$ is ample for ${\mathsf C}$ if
\begin{itemize}
\item[(Am1)] for every object $\cM\in {\mathsf C}$, there are positive integers
$r_1, \dots , r_p\in \N^+$ and an epimorphism $\bigoplus
_{i=1}^ps^{-r_i}\cO\to \cM$ in ${\mathsf C}$, and
\item[(Am2)] for every epimorphism $\cM\to \cN$ in ${\mathsf C}$, there is an
integer $n_0$ such that the induced map $\Hom _{\mathsf C}(s^{-n}\cO,
\cM)\to \Hom _{\mathsf C}(s^{-n}\cO, \cN)$ is surjective for every $n\geq
n_0$.
\end{itemize}
\end{dfn}

We define the graded algebra associated to an algebraic triple $({\mathsf C}, \cO, s)$ by
$B({\mathsf C}, \cO, s):=\bigoplus _{i\in \Z}\Hom_{\mathsf C}(\cO, s^i\cO)$.
Moreover, for any object $\cM\in {\mathsf C}$,
it is known that
$\bigoplus _{i\in \Z}\Hom_{\mathsf C}(\cO, s^i\cM)$ has a natural graded right $B({\mathsf C}, \cO, s)$-module structure.

\begin{thm} \textnormal{(\cite[Corollary 4.6 (1)]{AZ})} \label{thm.AZ}
Let $({\mathsf C}, \cO, s)$ be an algebraic triple.
If $\cO\in {\mathsf C}$ is a noetherian object, $\dim _k\Hom_{\mathsf C}(\cO, \cM)<\infty$ for all $\cM\in {\mathsf C}$, and $(\cO, s)$ is ample for ${\mathsf C}$, then
$B=B({\mathsf C}, \cO, s)_{\geq 0}$ is a right noetherian locally finite $\N$-graded algebra satisfying $\chi_1$, and
the functor
$${\mathsf C}\to \tails B; \quad \cF\mapsto \pi \left(\bigoplus _{i\in \N}\Hom_{\mathsf C}(\cO, s^i\cF)\right)$$
induces an equivalence of algebraic triples $({\mathsf C}, \cO, s)\to (\tails B, \cB, (1))$.
\end{thm}

Let $A$ be an $\N$-graded algebra.
Then the augmentation ideal $A_{\geq 1}$ is denoted by $\m$.
We define the functor
$\uG_{\m}: \GrMod A \to \GrMod A$ by
$\uG_{\m}(-)= \lim _{n \to \infty}\uHom_A(A/A_{\geq n}, -).$
The derived functor of $\uG_{\m}$ is denoted by $\RuG_{\m}(-)$, and
its cohomologies are denoted by $\uH^i_\m(-)= h^i(\RuG_{\m}(-))$.
For a graded module $M \in \GrMod A$, we define
\[
\depth M =\inf \{i\mid \uH^i_\m(M) \neq 0 \}\;\; \text{and} \;\; \ldim M=\sup \{i\mid \uH^i_\m(M) \neq 0\}.
\]
See \cite{Ydc}, \cite{Vet} for basic properties of $\RuG_\m(-)$.
%%\cite{Jlc},

If $A$ is an $\N$-graded algebra with $A_0=k$, then $A$ is called connected graded.
Note that a right noetherian connected graded algebra is locally finite.

\begin{dfn}
A two-sided noetherian connected graded algebra $A$ is called
AS-Gorenstein (resp.\, AS-regular) of dimension $d$ and of Gorenstein parameter $\ell$ if
\begin{itemize}
\item{} $\injdim_A A = \injdim_{A^{\op}} A= d <\infty$ (resp.\, $\gldim A = d<\infty$), and
\item{} $\uExt^i_A(k ,A) \cong \uExt^i_{A^{\op}}(k, A) \cong
\begin{cases}
k(\ell) & \text { if } i=d ,\\
0 & \text { if } i\neq d.
\end{cases}$
\end{itemize}
\end{dfn}

For $M \in \GrMod A$ and a graded algebra automorphism $\si \in \GrAut A$,
we define the twist $M_\si \in \GrMod A$ by $M_\si = M$ as a graded $k$-vector space
with the new right action $m*a = m\si(a)$.
Let $A$ be an AS-Gorenstein algebra.
Then it is well-known that $A$ has a balanced dualizing complex
$D\RuG_\m(A) \cong D\RuG_{\m^{\op}}(A) \cong A_\nu(-\ell)[d]$ in $\D{\GrMod A^e}$
with some graded algebra automorphism $\nu \in \GrAut A$.
This graded algebra automorphism $\nu \in \GrAut A$ is called the generalized Nakayama automorphism.
We define $\om_A := A_\nu(-\ell) \in \GrMod A^e$.

\begin{dfn}
Let $A$ be an AS-Gorenstein algebra of dimension $d$, and $M \in \grmod A$.
Then $M$ is called a graded maximal Cohen-Macaulay module if $\depth M = \ldim M=d$.
\end{dfn}

Let $A$ be an AS-Gorenstein algebra. Then $A$ is a graded maximal Cohen-Macaulay $A$-module.
It is well-known that $M \in \grmod A$ is graded maximal Cohen-Macaulay if and only if $\uExt^i_A(M,A)=0$ for all $i \neq 0$.
See \cite{Mcm} for basic properties of maximal Cohen-Macaulay modules.

We write $\CM A$ for the full subcategory of $\grmod A$ consisting of graded maximal Cohen-Macaulay modules.
If $M \in \grmod A$, then we define $M^\da = \uHom_A(M,A) \in \grmod A^{\op}$.
Similarly, if $N \in \grmod A^{\op}$, then we define $N^\da = \uHom_{A^{\op}}(N,A) \in \grmod A$.
It is well-known that the contravariant functors
\begin{align} \label{odual}
\xymatrix@C=4pc@R=1pc{ \CM A \ar@<0.6ex>[r]^{(-)^\da} &\CM {A^{\op}} \ar@<0.6ex>[l]^{(-)^\da}}
\end{align}
define a duality.
For $M \in \CM A$, we put $B=\uEnd_A(M), C=\uEnd_{A^{\op}}(M^\da)$. Then
\begin{align} \label{opiso}
 C \cong \uEnd_{A^{\op}}(M^\da) \cong \uEnd_{A}(M)^{\op} \cong B^{\op}
\end{align}
as graded algebras by the above duality.

The following theorem, called the maximal Cohen-Macaulay approximation theorem, plays a key role in this paper.

\begin{thm} \label{thm:CMapp}
Let $A$ be an AS-Gorenstein algebra.
For any $M \in \grmod A$, there exists a short exact sequence 
\[\xymatrix@C=1pc@R=1pc{
0 \ar[r] &L \ar[r] &Z \ar[r] & M \ar[r] &0} \]
in $\grmod A$ such that $Z \in \CM A$ and $\injdim_A L < \infty$.
Moreover $\uExt_A^i(X,L)=0$ holds for any $X \in \CM A$ and any $i\geq 1$.
\end{thm}

\begin{proof}
This follows from \cite[Proposition 5.3]{Mcm} and \cite[Lemma 3.5]{U6}.
\end{proof}

Recently, Minamoto and Mori \cite{MM} introduced the two notions of an $\N$-graded (not necessarily connected graded)
AS-regular algebra.

\begin{dfn} \cite[Definition 3.1]{MM}
A locally finite $\N$-graded algebra $B$ is called
AS-regular over $R=B_0$ of dimension $d$ and of Gorenstein parameter $\ell$ if
\begin{itemize}
\item{} $\gldim B = \gldim B^{\op} = d<\infty$, and
\item{} $\RuHom_B(B_0,B) \cong \RuHom_{B^{\op}}(B_0,B) \cong (DB_0)(\ell)[-d]$ in $\D{\GrMod B_0}$ and in $\D{\GrMod {B_0}^{\op}}$.
\end{itemize}
\end{dfn}

\begin{rem}
For the purpose of this paper, we do not require $\gldim B_0 < \infty$.
\end{rem}

\begin{dfn} \cite[Definition 3.9]{MM}
A locally finite $\N$-graded algebra $B$ is called
ASF-regular of dimension $d$ and of Gorenstein parameter $\ell$ if
\begin{itemize}
\item{} $\gldim B = \gldim B^{\op} = d<\infty$, and
\item{} $\RuG_{\m}(B) \cong \RuG_{\m^{\op}}(B) \cong (DB)(\ell)[-d]$ in $\D{\GrMod B}$ and in $\D{\GrMod B^{\op}}$.
\end{itemize}
\end{dfn}

By expanding the notion of an AS-regular algebra to $\N$-graded algebras,
Minamoto and Mori \cite{MM} gave a nice correspondence between $\N$-graded AS-regular algebras over $R$ of dimension $d$
with $\gldim R<\infty$ and quasi-Fano algebras of global dimension $d-1$.
This result provides a strong connection between noncommutative projective geometry
and representation theory of finite dimensional algebras. See \cite{HIO}, \cite{MM}, \cite{Mrm} for details.

At the end of this section, we give a comparison theorem for the two notions of AS-regular algebras.

\begin{lem} \label{lem:ld}
Let $B$ be a two-sided noetherian ASF-regular algebra, and $C$ a two-sided noetherian locally finite $\N$-graded algebra.
Then for any $M \in \GrMod C^{\op} \otimes B$,
\[ D\RuG_{\m}(M) \cong \RuHom_B(M,D\RuG_{\m}(B))  \]
in $\D{\GrMod B^{\op} \otimes C}$.
\end{lem}

\begin{proof}
Van den Bergh \cite{Vet} gave a theory on local duality for connected graded algebras.
One can check that the results in \cite[Sections 3--6]{Vet} hold with no essential
change for a noetherian locally finite $\N$-graded algebra
(see also \cite[Remark 3.6]{RRZ}, \cite[Lemma 3.2 (1)]{RRZ2}).
This follows from the locally finite $\N$-graded version of \cite[Theorem 5.1]{Vet}.
\end{proof}

\begin{thm} \label{thm:com}
If $B$ is a two-sided noetherian ASF-regular algebra of dimension $d$ and of Gorenstein parameter $\ell$,
then $B$ is an AS-regular algebra over $B_0$ of dimension $d$ and of Gorenstein parameter $\ell$.
\end{thm}

\begin{proof}
Since there exists an algebra automorphism $\mu \in \GrAut B$ such that $D\RuG_{\m}(B) \cong B_{\mu}(-\ell)[d]$ in 
$\D{\GrMod B^e}$,
we have 
\begin{align*}
\RuHom_B(B_0,B) &\cong \RuHom_B(B_0,B_{\mu}(-\ell)[d])_{\mu^{-1}}(\ell)[-d]\\
&\cong \RuHom_B(B_0,D\RuG_{\m}(B))_{\mu^{-1}}(\ell)[-d]\\
&\cong D\RuG_{\m}(B_0)_{\mu^{-1}}(\ell)[-d]\\
&\cong D(B_0)_{\mu^{-1}}(\ell)[-d]
\end{align*}
in $\D{\GrMod B^{\op} \otimes B_0}$, so $\RuHom_B(B_0,B) \cong D(B_0)(\ell)[-d]$ in $\D{\GrMod B_0}$ and in $\D{\GrMod {B_0}^{\op}}$.
Hence the result follows.
\end{proof}

\begin{cor} \label{cor:com}
Let $B$ be a two-sided noetherian locally finite $\N$-graded algebra with $\gldim B_0 < \infty$.
Then $A$ is ASF-regular if and only if AS-regular over $B_0$.
\end{cor}

\begin{proof}
This is a combination of \thmref{thm:com} and \cite[Theorem 3.12]{MM}.
\end{proof}

\section{Main Result}
In this section, we prove the main result (\thmref{thm:main}) and give an example of its use.
First we introduce a condition which we require for the main result.

\begin{dfn}
Let $A$ be an AS-Gorenstein algebra with the generalized Nakayama automorphism $\nu \in \GrAut A$.
Then $X \in \GrMod A$ is called $\nu$-stable if $X_\nu \cong X$ as graded right $A$-modules.
\end{dfn}
Since $A_\nu \cong A$ in $\GrMod A$, $A$ is always $\nu$-stable.
Clearly, if $A$ is symmetric, that is, the generalized Nakayama automorphism of $A$ is the identity,
then every $M \in \GrMod A$ is $\nu$-stable.

\begin{ex} \label{ex:st}
Let $S$ be an AS-regular algebra and $G$ a finite subgroup of $\GrAut S$ such that
$\fchar k$ does not divide $|G|$.
If $S^G$ is AS-Gorenstein, then $S \in \GrMod S^G$ is $\nu$-stable by \cite[Lemma 5.8]{U5}.
\end{ex}

\begin{lem} \label{lem:st}
Let $A$ be an AS-Gorenstein algebra and let $X \in \CM A$.
If $X$ is $\nu$-stable, then $X \cong X_{\nu^{-1}}$ in $\GrMod A$ and $X^\da$ is $\nu$-stable in $\GrMod A^{\op}$. 
\end{lem}

\begin{proof} It is easy to check that
$X \cong X_{\nu \nu^{-1}} \cong X_{\nu^{-1}}$ in $\GrMod A$, and
\begin{align*}
{_\nu}(X^\da)  
\cong \uHom_A(X,{{_\nu}}A)
\cong \uHom_A(X,A_{\nu^{-1}})
\cong \uHom_A(X_{\nu},A)
\cong \uHom_A(X,A) 
\cong X^\da
\end{align*}
in $\GrMod A^{\op}$.
\end{proof}

The notion of an $n$-cluster tilting module plays an important role in representation theory of orders,
especially, higher analogue of Auslander-Reiten theory. It can be regarded as a natural generalization of the classical
notion of Cohen-Macaulay representation-finiteness.

\begin{dfn}
Let $A$ be an AS-Gorenstein algebra. For a positive integer $n \in \N^{+}$,
a graded maximal Cohen-Macaulay module $X \in \CM A$ is called $n$-cluster tilting if
\begin{align*}
\add\{X(i)|i\in \Z\}
&= \{M \in \CM A | \uExt^i_A(X,M)=0 \;\text{for}\; 0 < i < n \}\\
&= \{M \in \CM A | \uExt^i_A(M,X)=0 \;\text{for}\; 0 < i < n \},
\end{align*}
where $\add\{X(i)|i\in \Z\}$ is the full subcategory of $\grmod A$
consisting of direct summands of finite direct sums of shifts of $X$.
\end{dfn}

Let $A$ be a noetherian locally finite $\N$-graded algebra. Then it is known that $\grmod A$ has the Krull-Schmidt property,
i.e., each finitely generated graded module is a direct sum of a uniquely determined set of indecomposable graded modules.
Recall that $M \in \grmod A$ is called basic if
each indecomposable direct summand occurs exactly once (up to isomorphism and degree shift of grading) in a direct sum decomposition.
The following proposition says that if $A$ has a $(d-1)$-cluster tilting module and $\gldim(\tails A)<\infty$,
then it has a $\nu$-stable $(d-1)$-cluster tilting module.

\begin{prop} \label{prop:bs}
Let $A$ be an AS-Gorenstein algebra of dimension $d\geq 2$, Gorenstein parameter $\ell$, and $\gldim(\tails A)<\infty$.
If $X \in \CM A$ is a basic $(d-1)$-cluster tilting module, then $X$ is $\nu$-stable.
\end{prop}

\begin{proof}
By \cite[Corollary 4.5]{U5}, we see that the stable category $\uCM A$ has the Serre functor 
$-\otimes_A A_{\nu}(-\ell)[d-1]$. Since 
$\uExt^{i}_A(X,X)\cong\bigoplus_{s \in \Z}\Hom_{\uCM A}(X,X(s)[i])=0$
for any $0<i<d-1$, we have $\uExt^{i}_A(X,X_\nu)=0$ for any $0<i<d-1$ by using the Serre functor of $\uCM A$,
so $X_\nu \in \add\{X(i)|i\in \Z\}$.
Since $X$ is basic, $X_\nu$ is also basic, so it follows that $X_\nu$ is a direct summand of $X$.
Similarly, we can show that $X_{\nu^{-1}}$ is a direct summand of $X$, so $X$ is a direct summand of $X_\nu$.
Hence the result follows.
\end{proof}

Next we prepare some lemmas which we need to prove the main result.

\begin{lem} \label{lem:Eiso}
Let $A$ be a graded algebra and $X$ a graded right $A$-module containing $A$ as a direct summand.
Then $X$ is a finitely generated graded left projective module over $\uEnd_A(X)$.
Moreover, for any $M \in \GrMod A$, there is a natural isomorphism
$$M \cong \uHom_A(X,M)\otimes_{\uEnd_A(X)} X$$
in $\GrMod A$.
\end{lem}

\begin{proof}
Put $B :=\uEnd_A(X)$.
Since $X \cong A \oplus Y$ for some $Y \in \GrMod A$, we have 
$$B= \uHom_A(X,X) \cong \uHom_A(A,X) \oplus \uHom_A(Y,X) \cong X \oplus \uHom_A(Y,X)$$
in $\GrMod B^{\op}$, so $X$ is finitely generated graded left projective over $B$.
Moreover one can verify that $\uEnd_{B^{\op}}(X) \cong A$ as graded algebras.
Thus, for any $M \in \GrMod A$, we have
\[ \uHom_A(X,M)\otimes_B X \cong \uHom_{A}(\uEnd_{B^{\op}}(X),M) \cong \uHom_{A}(A,M) \cong M \]
as graded right $A$-modules.
\end{proof}

\begin{lem} \label{lem:fg}
Let $A$ be an AS-Gorenstein algebra, and $X \in \CM A$ such that $X$ contains $A$ as a direct summand.
If $\uEnd_A(X)$ is right noetherian and $M \in \grmod A$, then $\uHom_A(X,M)$ is a finitely generated graded right $\uEnd_A(X)$-module.
\end{lem}

\begin{proof}
Put $B :=\uEnd_A(X)$. By \thmref{thm:CMapp}, we have an exact sequence
$$ 0 \to \uHom_A(X,L) \to \uHom_A(X,Z) \to \uHom_A(X,M) \to 0 $$
in $\GrMod B$ where $Z \in \CM A$ and $\injdim_A L<\infty$,
so it is enough to show that $\uHom_A(X,Z)$ is finitely generated.
Since $A$ is AS-Gorenstein and $Z \in \CM A$, we have a graded monomorphism $Z \to F$ in $\GrMod A$
where $F$ is a finitely generated free $A$-module.
Now $A$ is a direct summand of $X$, so there exist graded monomorphisms $Z \to \hat{X}$ in $\GrMod A$ and 
$\uHom_A(X,Z) \to \uHom_A(X,\hat{X})$ in $\GrMod B$ where $\hat{X}$ is a finite direct sum of shifts of $X$.
Since $\uHom_A(X,\hat{X})$ is finitely generated and $B$ is right noetherian, $\uHom_A(X,Z)$ is also finitely generated.
\end{proof}

\begin{lem} \label{lem:lc}
Let $B$ be a two-sided noetherian locally finite $\N$-graded algebra. For any $M \in \GrMod B^e$,
\[ \RuG_{\m}(\RuG_{\m^{\op}}(M)) \cong \RuG_{\m^{\op}}(\RuG_{\m}(M)) \]
in $\D{\GrMod B^e}$.
\end{lem}

\begin{proof}
One can show the locally finite $\N$-graded version of \cite[Lemma 4.5]{Vet},
so the assertion holds.
\end{proof}

\begin{lem} \label{lem:ld2}
Let $B$ be a two-sided noetherian ASF-regular algebra with an idempotent $e$.
If $M \in \grmod B$ is finite dimensional over $k$ such that $Me=0$, then
$\uExt^i_B(M,eB) = 0$
for any $i$.
\end{lem}

\begin{proof}
Since  $B$ is ASF-regular, we have
\begin{align*}
\RuHom_B(M,eB) \cong e\RuHom_B(M,D\RuG_{\m}(B)(\ell)[-d]) \cong e\RuHom_B(M,D\RuG_{\m}(B))(\ell)[-d].
\end{align*}
Moreover, by \lemref{lem:ld},
\begin{align*}
e\RuHom_B(M,D\RuG_{\m}(B))(\ell)[-d] \cong eD\RuG_{\m}(M)(\ell)[-d].
\end{align*}
Since $M \in \grmod B$ is finite dimensional over $k$, $\RuG_{\m}(M) \cong M$, so we have
$\uExt^i_B(M,eB) \cong eD\uH^{d-i}_{\m}(M)(\ell)=0$
for all $i \neq d$, and 
$\uExt^d_B(M,eB) \cong eD(M)(\ell) \cong D(Me)(\ell)=0$ because $Me=0$.
\end{proof}

Now we are ready to prove the main result of this paper.

\begin{thm} \label{thm:main}
Let $A$ be a two-sided noetherian connected graded algebra satisfying $\chi$ on both sides, and let $X \in \grmod A$. 
We consider the following conditions {\bf (A)} and {\bf (B)}:

{\bf (A):} $A$ and $X$ satisfy
\begin{description}
\item[(A1)] $A$ is an AS-Gorenstein algebra of dimension $d\geq 2$,
\item[(A2)] $X$ is a $(d-1)$-cluster tilting module,
\item[(A3)] $\uEnd_A(X)_{<0}=0$, and
\item[(A4)] $\uExt^1_A(X,M)$ and $\uExt^1_A(M,X)$ are finite dimensional over $k$ for any $M \in \CM A$.
\end{description}

{\bf (B):} $B := \uEnd_{A}(X)$ satisfies
\begin{description}
\item[(B1)] $B$ is a two-sided noetherian ASF-regular algebra of dimension $d \geq 2$,
\item[(B2)] $-\otimes_B X$ induces an equivalence functor $\tails B \xrightarrow{\sim} \tails A$, and 
\item[(B3)] $X^\da \otimes_B -$ induces an equivalence functor $\tails B^{\op} \xrightarrow{\sim} \tails A^{\op}$.
\end{description}
Then 
\begin{enumerate}
\item If {\bf (A)} is fulfilled and $X$ is either $\nu$-stable or basic, then {\bf (B)} holds.
\item If {\bf (B)} is fulfilled and $X$ contains $A$ as a direct summand, then {\bf (A)} holds.
\end{enumerate}
\end{thm}

\vspace{2truemm} \noindent
{\bf Proof of $(1)$ in \thmref{thm:main}.}
Suppose that $A$ and $X$ satisfy {\bf (A)}, and $X$ is either $\nu$-stable or basic.
Since $A$ is indecomposable, {\bf (A2)} implies that $X$ is a graded maximal Cohen-Macaulay $A$-module containing a shift of $A$ as a direct summand.
By properties of degree shifts, we may assume that $X$ contains $A$ as a direct summand without loss of generality.

The proof is divided into the following several steps:
\begin{enumerate}
\item[(a)] $(\cX,(1))$ is ample for $\tails A$,
\item[(b)] $B$ is right noetherian, and {\bf (B2)} holds,
\item[(c)] $\gldim B = d$,
\item[(d)] $\uH_{\m}^{i}(B)=0$ for any $i \neq d$,
\item[(e)] $(\cX^\da,(1))$ is ample for $\tails A^{\op}$,
\item[(f)] $B$ is left noetherian, and {\bf (B3)} holds,
\item[(g)] $\gldim B^{\op} = d$,
\item[(h)] $\uH_{\m^{\op}}^{i}(B)=0$ for any $i \neq d$, and
\item[(i)] $\uH_{\m}^{d}(B) \cong \uH_{\m^{\op}}^{d}(B) \cong (DB)(\ell)$ in $\GrMod B$ and in $\GrMod B^{\op}$.
\end{enumerate}

\vspace{2truemm} 
{\bf Proof of (a)}:
We show that $(\cX,(1))$ is ample for $\tails A$.
Since $A$ is a direct summand of $X$, it is easy to check that the condition (Am1) is satisfied.
Let $f:\cM \to \cN$ be an epimorphism in $\tails A$.
It gives a short exact sequence
$0 \to \cK \to \cM \to \cN \to 0$ in $\tails A$.
We take a finitely generated graded module $K$ such that $\pi K=\cK$.
By \thmref{thm:CMapp}, there exists an exact sequence
\[\xymatrix@C=1pc@R=1pc{
0 \ar[r] &L \ar[r] &Z \ar[r] & K \ar[r] &0} \]
in $\grmod A$ such that $Z \in \CM A$ and $\injdim_A L<\infty$.
Furthermore \thmref{thm:CMapp} also implies that $\uExt^1_A(X,K) \cong \uExt^1_A(X,Z)$ holds, so $\uExt^1_A(X,K)$ is finite dimensional over $k$ by {\bf (A4)}.
By \cite [Corollary 7.3 (2)]{AZ}, it follows that $\uExt^1_{\cA}(\cX, \cK)$ is right bounded,
and thus we have $\Ext^1_{\cA}(\cX(-n), \cK)=\uExt^1_{\cA}(\cX, \cK)_n=0$ for all $n\gg 0$.  Since 
$$\Hom_{\cA}(\cX(-n), \cM)\to \Hom_{\cA}(\cX(-n), \cN)\to \Ext^1_{\cA}(\cX(-n), \cK)$$ is exact,
$\Hom_{\cA}(\cX(-n), \cM)\to \Hom_{\cA}(\cX(-n), \cN)$ is surjective for all $n\gg 0$, so the condition (Am2) is also satisfied, hence $(\cX, (1))$ is ample for $\tails A$.

\vspace{2truemm} 
{\bf Proof of (b)}:
We show that $B$ is right noetherian and {\bf (B2)} is satisfied.
The basic idea of the proof comes from \cite[Section 2]{MU1}.
Since $\depth_A X=d\geq 2$, we have 
$$B= \uEnd_A(X)=B(\grmod A, X, (1))\cong B(\tails A, \cX, (1))$$
by \cite [Lemma 3.3]{Mmc}.
By using {\bf (A3)} and \thmref{thm.AZ}, it follows that $B = B_{\geq 0}\cong  B(\tails A, \cX, (1))_{\geq 0}$
is right noetherian locally finite.
Moreover the functor 
$$ F:= \pi \circ \uHom_{\cA}(\cX, -):\tails A\to \tails B$$ is an equivalence.
For $M\in \grmod A$, there exists $n\in \Z$ such that
$$\uHom_{\cA}(\cX, \cM)_{\geq n}\cong \uHom_A(X, M)_{\geq n}$$
in $\grmod B$ by \cite [Corollary 7.3 (2)]{AZ}, so the functor $F$ is induced by the functor $\uHom_A(X, -):\grmod A\to \grmod B$.
By using \lemref{lem:Eiso}, we see that the functor $\tails B \to \tails A$ induced by $- \otimes_B X:\grmod B\to \grmod A$ is an equivalence functor quasi-inverse to $F$.

\vspace{2truemm}
{\bf Proof of (c)}:
Here we show that $\gldim B = d$.
First let us explain that we can construct a graded right $\add\{X(i)|i\in \Z\}$-approximation of $M \in \grmod A$.
Since $\uHom_A(X,M)$ is a finitely generated graded right $B$-module by \lemref{lem:fg}, we can take 
$f_i \in \Hom_{\GrMod A} (X(s_i),M)$ such that $f_1,\cdots,f_n$ generate $\uHom_{A} (X,M)$.
Thus for any $f \in \Hom_{\GrMod A} (X(t),M)$, there exist graded homomorphisms $g_i \in \Hom_{\GrMod A} (X(t), X(s_i))$ such that
\[\xymatrix@C=1pc@R=1pc{
X(s_1) \oplus \cdots \oplus X(s_n) \ar[rrrr]^(0.6){\left(\begin{smallmatrix} f_1 &\dots &f_n \end{smallmatrix}\right)} &&&&M\\
&& X(t) \ar[rru]_{f} \ar[llu]^{\left(\begin{smallmatrix} g_1 \\ \vdots \\g_n \end{smallmatrix}\right)}
} \]
commutes. We can check that $\phi := \begin{pmatrix} f_1 &\dots &f_n \end{pmatrix} : \bigoplus_{i=1}^{n}X(s_i) \to M$ is surjective, and 
\[\xymatrix@C=2pc@R=1pc{
\uHom_{A}(X,\bigoplus_{i=1}^{n}X(s_i))  \ar[rr]^(0.55){\uHom(X,\phi)} && \uHom_{A}(X,M)
}\]
is also surjective.

By the above arguments, the proof of $\gldim B \leq d$ goes along the same line as that of \cite[Theorem 3.6]{DH} (see also \cite[Theorem 5.10]{U5}).
Let $N \in \grmod B$ and
take a projective presentation $P_1\to P_0 \to N \to 0$.
Since we can write $P_i \cong \uHom_{A}(X,X_i)$ where $X_i \in \add\{X(i)|i\in \Z\}$ for each $i$,
we have an exact sequence
\[\xymatrix@C=1.2pc@R=1pc{
0 \ar[r] &\uHom_{A}(X,M_1) \ar[r] &\uHom_{A}(X,X_1) \ar[r] &\uHom_{A}(X,X_0) \ar[r] & N \ar[r] &0} \]
in $\grmod B$ such that $0\to M_1\to X_1\to X_0$ is exact in $\grmod A$.
By using a graded right $\add\{X(i)|i\in \Z\}$-approximation, we have a module $X_2 \in \add\{X(i)|i\in \Z\}$ and a surjection $X_2 \to M_1$ such that 
$\uHom_{A}(X,X_2) \to \uHom_{A}(X,M_1)$ is also surjective. Let $M_2$ the kernel of $X_2 \to M_1$. Then it is easy to see that $\uExt^1_A(X,M_2)=0$.
Continuing in this way inductively, we can make exact sequences
\begin{align*} 
\xymatrix@C=1.2pc@R=1pc{
0\ar[r]&M_{d-1}\ar[r]&X_{d-1}\ar[r] &\cdots \ar[r] &X_2 \ar[r]&X_1 \ar[r]^{\xi} &X_0 \ar[r] & \Coker\xi \ar[r] &0
}
\end{align*}
in $\grmod A$, and
\begin{align*}
\xymatrix@C=1.2pc@R=1pc{
0 \ar[r] &\uHom_{A}(X,M_{d-1})\ar[r]& P_{d-1} \ar[r] &\cdots \ar[r] &P_2\ar[r] &P_1 \ar[r] &P_0 \ar[r] &N \ar[r] &0
}
\end{align*}
in $\grmod B$ where $P_i = \uHom_{A}(X,X_i)$. Furthermore we see $\uExt^{j}_A(X,M_i)=0$ for any $2\leq i\leq d-1$ and any $0<j<i$.
If $M_{d-1}=0$, then clearly $\projdim_B N \leq d-1$.
We now consider the case $M_{d-1} \neq 0$.
Since $X_i \in \CM A$, it follows from the depth lemma (cf. \cite[Proposition 1.2.9]{B}) that $M_{d-1}$ is graded maximal Cohen-Macaulay.
Moreover, since $\uExt^{j}_A(X,M_{d-1})=0$ for $0<j<d-1$, it follows that $M_{d-1} \in \add\{X(i)|i\in \Z\}$ by {\bf (A2)}.
Thus we obtain $\projdim_B N \leq d$.

To see that $\gldim B = d$, consider a graded projective resolution of $\uHom_A(X,k)$ in $\grmod B$, namely,
\begin{align*}
\xymatrix@C=1.2pc@R=1pc{
\cdots \ar[r] &P_2\ar[r] &P_1 \ar[r] &P_0 \ar[r] &\uHom_A(X,k) \ar[r] &0.
}
\end{align*}
Applying $- \otimes_B X$ to the above exact sequence, we have an exact sequence
\begin{align*}
\xymatrix@C=1.2pc@R=1pc{
\cdots \ar[r] &P_2\otimes_B X \ar[r] &P_1\otimes_B X \ar[r] &P_0\otimes_B X \ar[r] &\uHom_A(X,k)\otimes_B X \ar[r] &0
}
\end{align*}
in $\grmod A$. By \lemref{lem:Eiso}, $\uHom_A(X,k)\otimes_B X \cong k$.
Since each $P_i \otimes_B X$ is a direct summand of finite direct sums of shifts of $X$,
it is graded maximal Cohen-Macaulay,
so $\projdim_B \uHom_A(X,k) \leq d-1$ gives a contradiction to the fact that $\depth_A k=0$. Thus $\gldim B=d$.

\vspace{2truemm}
{\bf Proof of (d)}:
We next show that $\uH_{\m}^{i}(B)=0$ for $i \neq d$.
By the arguments in the proof of (a) and (b), we see that
\[\xymatrix@C=2pc@R=2pc{
(\grmod A, X, (1)) \ar[r]^{\pi} & (\tails A, \cX, (1)) \ar[d]^{F}_{\cong} \\
(\grmod B, B, (1)) \ar[r]^{\pi} \ar[u]^{-\otimes_B X} & (\tails B, \cB, (1)),
}\]
commutes, and the graded algebra homomorphism
$$\uEnd_B(B) = B(\grmod B, B, (1)) \to B(\tails B, \cB, (1)) = \uEnd_\cB(\cB)$$
induced by the natural functor $\pi$ is an isomorphism.
This says that the natural map $\vph$ appearing in the exact sequence
\[ \xymatrix@C=1pc@R=1pc{
0 \ar[r]& \uH_{\m}^{0}(B) \ar[r]& \uEnd_B(B) \ar[r]^{\vph}& \uEnd_\cB(\cB) \ar[r]& \uH_{\m}^{1}(B) \ar[r]& 0 
}\]
in \cite[Proposition 7.2 (2)]{AZ} is an isomorphism. Thus $\uH_{\m}^{0}(B) = \uH_{\m}^{1}(B) =0$.

Since we already have the equivalence functor in {\bf (B2)}, it follows that 
\begin{align} \label{eqb4}
\uExt^i_\cB(\cB,\cB) \cong \uExt^i_\cA(\cX,\cX)
\end{align}
for any $i$.
Using $\depth_A X = d$ and {\bf (A2)}, we have 
$\uExt^i_\cA(\cX,\cX) \cong \uExt^i_A(X,X)=0$
for any $1 \leq i\leq d-2$, so $\uExt^i_\cB(\cB,\cB)=0$ for any $1 \leq i\leq d-2$.
Thus $\uH_{\m}^{i}(B)=0$ for $2 \leq i\leq d-1$ by \cite[Theorem 7.2 (2)]{AZ}.
Furthermore $B$ has global dimension $d$ by (c),
so $\uH_{\m}^{i}(B)=0$ for $i\geq  d+1$.

\vspace{2truemm} 
{\bf Proof of (e)}: By the duality (\ref{odual}), we see that {\bf (A)} is equivalent to the following:

{\bf (A$^{\op}$):} $A^{\op}$ and $X^\da$ satisfy
\begin{description}
\item[(A1$^{\op}$)] $A^{\op}$ is an AS-Gorenstein algebra of dimension $d\geq 2$ 
\item[(A2$^{\op}$)] $X^\da$ is a $(d-1)$-cluster tilting module, and
\item[(A3$^{\op}$)] $\uEnd_{A^{\op}}(X^\da)_{<0}=0$,
\item[(A4$^{\op}$)] $\uExt^1_{A^{\op}}(N,X^\da)$ and $\uExt^1_{A^{\op}}(X^\da,N)$ are finite dimensional over $k$ for any $N \in \CM {A^{\op}}$.
\end{description}
Hence the same argument in the proof of (a) implies that $(X^\da,(1))$ is ample for $\tails A^{\op}$.

\vspace{2truemm} 
{\bf Proof of (f)}:
Since $(X^\da,(1))$ is ample for $\tails A^{\op}$, It follows from the same argument in the proof of (b) that
$C := \uEnd_{A^{\op}}(X^\da)$ is right noetherian and $- \otimes_C X^\da$ induces
an equivalence functor $\tails C \xrightarrow{\sim} \tails A^{\op}$. However, $B^{\op} \cong C$ as graded algebras by
(\ref{opiso}), so we obtain that $B$ is left noetherian and $X^\da \otimes_B -$ induces
an equivalence functor $\tails B^{\op} \xrightarrow{\sim} \tails A^{\op}$.

\vspace{2truemm}
{\bf Proof of (g) and (h)}:
By the arguments in the proofs of (e) and (f), 
the proof of (g) (resp. (h)) is obtained by the same way as that of (c) (resp. (d)).

\vspace{2truemm}
{\bf Proof of (i)}: 
By (b) and (c), $\tails A \cong \tails B$ and $\gldim B < \infty$, so it follows that $\tails A$ has finite global dimension.
Thus the derived category $\Db{\tails A}$ has the Serre functor
$-\otimes_\cA \om_\cA[d-1] \cong -\otimes_\cA \cA_\nu(-\ell)[d-1] $ by \cite[Theorem A.4]{dNvB}.
Moreover, if $X$ is basic, then it is $\nu$-stable by \propref{prop:bs}.
Using the equivalence in {\bf (B2)} and the fact that $X$ is $\nu$-stable, for any $\cM= \pi M \in \tails B$,
we have natural isomorphisms
\begin{align*}
\uHom_\cB(\cM,\cB(-\ell ))
&\cong \uHom_\cA(\pi (M\otimes_B X) ,\pi X(-\ell ))\\
&\cong \uHom_\cA(\pi (M\otimes_B X_{\nu^{-1}}) ,\pi X_{\nu^{-1}}(-\ell ))\\
&\cong \uHom_\cA(\pi (M\otimes_B X_{\nu^{-1}}) ,\pi X(-\ell )) &&(\text{by \lemref{lem:st}})\\
&\cong D\uExt^{d-1}_\cA(\pi X(-\ell ), \pi (M\otimes_B X_{\nu^{-1}}\otimes_A A_{\nu}(-\ell ) ))\\
&\cong D\uExt^{d-1}_\cA(\pi X(-\ell ), \pi (M\otimes_B X)(-\ell ))\\
&\cong D\uExt^{d-1}_\cB(\cB, \cM)
\end{align*}
as graded $k$-vector spaces. It follows that $\cB(-\ell )$ is a dualizing sheaf of $\tails B$ in the sense of \cite{YZ}.
By \cite[Proposition 7.10 (3)]{AZ}, it is easy to check that $\cd(\tails B) =d-1$. 
Moreover, by \thmref{thm.AZ}, we see that $B$ is a right noetherian locally finite graded algebra satisfying $\chi_1$,
so it follows that
\[ \Hom_\cB(\cB, \cB(-\ell)) \cong D\uExt^{d-1}_\cB(\cB, \cB) \]
in $\GrMod B$ by the proof of \cite[Theorem 2.3 (2)]{YZ}. Hence we obtain
\begin{align} \label{eq:lc1}
 B(-\ell) \cong \Hom_\cB(\cB, \cB)(-\ell) \cong D\uExt^{d-1}_\cB(\cB, \cB) \cong D\uH_{\m}^{d}(B) 
\end{align}
in $\GrMod B$ by (d). Dually, (f), (g), (h), and \lemref{lem:st} imply
\begin{align}\label{eq:lc2}
 B(-\ell) \cong D\uH_{\m^{\op}}^{d}(B)
\end{align}
in $\GrMod B^{\op}$.
In addition, we see that $\uH_{\m}^{d}(B)$ and $\uH_{\m^{\op}}^{d}(B)$ are right bounded graded $B^e$-modules,
so it follows from \lemref{lem:lc} that
\begin{align*}
\uH_{\m^{\op}}^{d}(B)
&\cong \RuG_{\m}(\uH_{\m^{\op}}^{d}(B))
\cong \RuG_{\m}(\uH_{\m^{\op}}^{d}(B)[-d])[d]\\
&\cong \RuG_{\m}(\RuG_{\m^{\op}}(B))[d]
\cong \RuG_{\m^{\op}}(\RuG_{\m}(B))[d]\\
&\cong \RuG_{\m^{\op}}(\uH_{\m}^{d}(B)[-d])[d]
\cong \RuG_{\m^{\op}}(\uH_{\m}^{d}(B))
\cong \uH_{\m}^{d}(B)
\end{align*}
in $\D{\GrMod B^e}$. Therefore our assertion follows from (\ref{eq:lc1}) and (\ref{eq:lc2}).

\vspace{2truemm}
Hence the proof of \thmref{thm:main} (1)  is now complete.\qed

\vspace{2truemm} \noindent
{\bf Proof of $(2)$ in \thmref{thm:main}.}
Suppose that $X$ satisfies {\bf (B)}, and contains $A$ as a direct summand.
Clearly {\bf (A3)} is satisfied by {\bf (B1)}.

First we show that {\bf (A1)} holds.
Since $A$ is a direct summand of $X$, there exists an idempotent $e \in B$ such that
$eBe \cong \uEnd_A(A) \cong A$ as graded algebras and $Be \cong \uHom_A(A,X) \cong X$ as graded $B$-$A$ bimodules. 
Then {\bf (B2)} says that the functor $\tails B \to \tails eBe$ induced by $-\otimes_B Be$ is an equivalence, so it follows that
$B/(e)$ is finite dimensional over $k$ by \cite[Lemma 3.17]{MU2}.
Let $\ro$ be the composition map 
$$Be \Lotimes_{eBe}eB \xrightarrow[\text{nat.}]{} Be \otimes_{eBe}eB \xrightarrow[\text{mult.}]{} B.$$
We have the triangle
\[Be \Lotimes_{eBe}eB \xrightarrow{\ro} B \to C \to Be \Lotimes_{eBe}eB[-1]\]
in $\D{\grmod B}$. Applying $-\Lotimes_{B}Be$ implies $C\Lotimes_{B}Be=0$.
It follows that $h^i(C)e = 0$ for all $i$, and thus $\uExt^j_B(h^i(C), eB)=0$ for all $i,j$ by \lemref{lem:ld2}.
By using a hypercohomology spectral sequence \cite[5.7.9]{Wei}, we obtain $h^p(\RuHom_B(C, eB))=0$ for all $p$, so $\RuHom_B(C, eB)=0$.
Applying $\RuHom_B(-, eB)$ to the above triangle induces
\[ \RuHom_B(Be \Lotimes_{eBe}eB, eB) \cong \RuHom_B(B, eB) \cong eB.\] 
On the other hand, we have 
\begin{align*}
\RuHom_B(Be \Lotimes_{eBe}eB, eB) \cong \RuHom_{eBe}(Be, \RuHom_B(eB, eB)) \cong \RuHom_{eBe}(Be, eBe),
\end{align*}
so it follows that 
\begin{align} \label{eq:CM}
\uExt^i_A(X,A)\cong \uExt^i_{eBe}(Be, eBe)=0
\end{align}
for all $i>0$.

Let $M \in \grmod A$. Since $\gldim B=d$, we can take a projective resolution
\begin{align*}
\xymatrix@C=1.2pc@R=1pc{
0 \ar[r] &P_m\ar[r] & \cdots \ar[r] &P_2\ar[r] &P_1 \ar[r] &P_0 \ar[r] &M\otimes_{eBe}eB \ar[r] &0
}
\end{align*}
in $\grmod B$ with $m \leq d$.
Applying $-\otimes_B Be$ to the above exact sequence, we have an exact sequence
\begin{align*}
\xymatrix@C=1.2pc@R=1pc{
0 \ar[r] &P_m\otimes_B Be\ar[r] & \cdots \ar[r] &P_2\otimes_B Be\ar[r] &P_1\otimes_B Be \ar[r] &P_0\otimes_B Be \ar[r] &M \ar[r] &0.
}
\end{align*}
Since each $P_j\otimes_B Be \in \add\{Be(i)|i \in \Z\} \cong \add\{X(i)|i \in \Z\}$, By (\ref{eq:CM}), we see that $\uExt_A^{m+1}(M,A)=0$.
Thus $\injdim_A A\leq d$.
By {\bf (B1)}, {\bf (B2)}, and {\bf (B3)}, $\gldim(\tails A)< \infty$ and $\gldim(\tails A^{\op})< \infty$.
Moreover $A$ satisfies $\chi$ on both sides, so it has a dualizing complex by \cite[Theorem 6.3]{Vet}.
It follows from \cite[Theorem 3.6]{DW} that $A$ is AS-Gorenstein.
If $\injdim_A A\leq d-1$, then $\gldim (\tails B) = \gldim (\tails A) \leq d-2$ by the Serre duality \cite[Theorem A.4]{dNvB}.
This is a contradiction to the fact that $\uExt^{d-1}_\cB(\cB,\cB) \cong \uH_{\m}^d(B) \neq 0$.
Hence $A$ is AS-Gorenstein of dimension $d$.

To show that {\bf (A2)} holds, it is enough to show that
\begin{enumerate}
\item[(a)] $X \in \CM A$ and $\uExt_{A}^i(X,X)=0$ for any $0<i<d-1$, 
\item[(b)] $M \in \CM A$ satisfying $\uExt^i_{A}(X,M)=0$ for any $0<i<d-1$ belongs to $\add\{X(i)|i\in \Z\}$, and
\item[(c)] $M \in \CM A$ satisfying $\uExt^i_{A}(M,X)=0$ for any $0<i<d-1$ belongs to $\add\{X(i)|i\in \Z\}$.
\end{enumerate}
By (\ref{eq:CM}), we see that $X \in \CM A$.
By {\bf (B2)}, the functor $\tails B \to \tails A$ induced by $- \otimes_A X$ is an equivalence,
so $\uExt^i_\cA(\cX,\cX) \cong \uExt^i_\cB(\cB,\cB)$ for any $i$.
Using $\depth_A X = d$ and \cite[Theorem 7.2 (2)]{AZ}, it follows that 
$$\uExt^i_A(X,X) \cong \uExt^i_\cA(\cX,\cX) \cong \uExt^i_\cB(\cB,\cB) \cong \uH_{\m}^{i+1}(B)$$
for any $0 < i < d-1$.
Hence (a) holds by {\bf (B1)}.

We now give the proof of (b). Let $M \in \CM {A}$ be such that $\uExt^i_{A}(X,M)=0$ for any $0<i<d-1$.
Since we know that $A$ is AS-Gorenstein, taking a free resolution of $M^{\da}$ in $\grmod A^{\op}$ and applying $(-)^\da$,
we have an exact sequence 
\begin{align*}
\xymatrix@C=1pc@R=1pc{
0\ar[r]&M \ar[r]& F_0 \ar[r] &\cdots  \ar[r]&F_{d-3} \ar[r] &Y \ar[r] &0
}
\end{align*}
in $\grmod A$ where each $F_i$ is a graded free $A$-module, and $Y \in \CM A$.
Then we can make an exact sequence 
\begin{align*}
\xymatrix@C=1pc@R=1pc{
0\ar[r]&\uHom_A(X, M) \ar[r]&\uHom_A(X, F_0) \ar[r] &\cdots  \ar[r]&\uHom_A(X, F_{d-3}) \ar[r] &\uHom_A(X, Y) \ar[r] &0
}
\end{align*}
in $\GrMod B$ because $\uExt^i_{A}(X,M)=0$ for $0<i<d-1$.
Moreover, taking a free presentation of $Y^\da \in \CM{A^{\op}}$ and applying $\Hom_A(X, (-)^\da)$,
we have an exact sequence
\[ \xymatrix@C=1pc@R=1pc{ 0\ar[r]&\Hom_A(X, Y) \ar[r]&\Hom_A(X, F_{d-2}) \ar[r]&\Hom_A(X, F_{d-1})} \]
in $\GrMod B$ where $F_{d-2}$ and $F_{d-1}$ are graded free $A$-modules.
By the assumption that $A$ is a direct summand of $X$, each $\uHom_A(X, F_i)$ is a graded right projective $B$-module. 
Since any graded right $B$-module has projective dimension at most $d$ by {\bf (B1)},
it holds that $\projdim_B \uHom_A(X, M)=0$.
This means that $\Hom_A(X, M)$ is finitely generated graded right projective over $B$ by \lemref{lem:fg},
so $M \in \add\{X(i)|i\in \Z\}$ by \lemref{lem:Eiso}.
Hence (b) is proved.

We next give the proof of (c). Let $M \in \CM {A}$ be such that $\uExt^i_{A}(M,X)=0$ for any $0<i<d-1$.
Then $\uExt^i_{A^{\op}}(X^\da,M^\da)=0$ for any $0<i<d-1$.
Since $B^{\op} \cong \uEnd_{A^{\op}}(X^\da)$ is also ASF-regular,
the same method in the proof of (b) yields $M^\da \in \add\{X^\da(i)|i\in \Z\}$.
Thus we see $M \in \add\{X(i)|i\in \Z\}$ and we obtain (c).

In addition, since we have $\gldim B< \infty$ and $\tails B \cong \tails A$ by {\bf (B1)} and {\bf (B2)},
it follows that $\gldim(\tails A)< \infty$,
so {\bf (A4)} is satisfied by \cite[Lemma 5.7]{U5}.

Hence the proof of \thmref{thm:main} (2) is finished.\qed

\begin{rem} \label{rem:A4}
By observing the proof of \thmref{thm:main}, we notice that the condition {\bf (A4)} in \thmref{thm:main} can be replaced by the condition
$$ \text{\bf (A4'):} \gldim(\tails A) < \infty.$$
In this sense, it can be said that {\bf (A4)} corresponds to the graded isolated singularity property.
\end{rem}

Here we obtain the following result.

\begin{cor} \label{cor:main}
Let $A$ be an AS-Gorenstein algebra of dimension $d\geq 2$ and of Gorenstein parameter $\ell$.
Assume that $\gldim(\tails A)< \infty$ and $A$ has a $(d-1)$-cluster tilting module $X \in \CM {A}$ satisfying $\uEnd_A(X)_{<0}=0$.
Then a basic $(d-1)$-cluster tilting module $Y$ can be extracted from $X$,
and besides $B=\uEnd_A(Y)$ is a two-sided noetherian AS-regular algebra over $B_0$ of dimension $d$ and of Gorenstein parameter $\ell$ such that $\tails B \cong \tails A$.
\end{cor}

\begin{proof}
By the proof of \thmref{thm:main}(1), we see that Gorenstein parameter of $B$ is given by Gorenstein parameter of $A$, 
so the statement follows from \thmref{thm:main}, \remref{rem:A4} and \thmref{thm:com}.
\end{proof}

It is clear that {\bf (A2)}, {\bf (A3)}, {\bf (A4)} and the $\nu$-stable assumption are conditions for a ``right'' $A$-module $X$.
On the other hand, we see that {\bf (B1)}, {\bf (B2)} and {\bf (B3)} are ``two-sided'' conditions for $B = \uEnd_A(X)$.
Thus \thmref{thm:main} (1) asserts that one-sided conditions for $X$ imply two-sided conditions for $B=\uEnd_A(X)$.

We now consider the case that $A$ is a noncommutative quotient singularity.

\begin{ex}
Let $S$ be an AS-regular algebra of dimension $d \geq 2$.
Let $G$ be a finite subgroup of $\GrAut S$ such that $\hdet g =1$ for each $g \in G$ (in the sense of J\o rgensen and Zhang \cite{JZ}), and $\fchar k$ does not divide $|G|$. 
Assume that $S*G/(e)$ is finite dimensional over $k$
where $e= \frac{1}{|G|} \sum_{g \in G} 1*g \in S*G$.
Then
\begin{itemize}
\item $S^G$ is AS-Gorenstein of dimension $d \geq 2$ by \cite[Theorem 3.3]{JZ},
\item $S \in \CM{S^G}$ is a $(d-1)$-cluster tilting module by \cite[Theorems 3.10, 3.15]{MU1},
\item $\uEnd_{S^G}(S)_{<0} = (S*G)_{<0} =0$ by \cite[Theorem 3.10]{MU1},
\item $\gldim(\tails S^G) < \infty$ by \cite[Theorem 3.10, Lemma 2.12]{MU1}, and
\item $S$ is a $\nu$-stable $S^G$-module by \exref{ex:st},
\end{itemize}
so the assumption of \thmref{thm:main} (1) is satisfied. If fact, it follows from \cite[Corollary 3.6, Theorem 3.10]{MU1} that
$B=\uEnd_{S^G}(S)$ is a two-sided noetherian AS-regular algebra over $B_0 \cong kG$ of dimension $d$ such that $\tails \uEnd_{S^G}(S) \cong \tails S^G$.
Hence \thmref{thm:main} is regarded as a detailed version of this phenomenon.
\end{ex}

For the rest of this paper, we construct a noncommutative quadric hypersurface which gives a concrete example of \thmref{thm:main}. 
See \cite[Section 5]{SV} and \cite[Section 4]{U6} for detailed information.

\begin{ex}
In this example, we assume that $k$ is algebraically closed of characteristic $0$.
Let
$$S= k\ang{x,y,z}/(xy+yx-z^2, xz+zx, yz+zy), \quad \deg x=\deg y=\deg z=1.$$
Then $S$ is a Koszul AS-regular algebra of dimension $d_S=3$ and Gorenstein parameter $\ell_S=3$ with a central regular element $x^2+y^2 \in S_2$.
Let $$A=S/(x^2+y^2).$$
Then $A$ is a Koszul AS-Gorenstein algebra of dimension $d_A=2$ and Gorenstein parameter $\ell_A=1$.
Take $w:= x^2 \in A^!_2$ so that $S^! \cong A^!/(w)$.
We define a finite dimensional algebra $C(A)$ by $C(A):= A^![w^{-1}]_0$ (see \cite[Lemma 5.1]{SV}).
It is easy to check that $C(A) \cong k[t]/(t^4-1) \cong k^4$ as algebras.
By \cite[Proposition 4.1]{U6}, $A$ is of finite Cohen-Macaulay representation type,
so it has a $1$-cluster tilting module.
It follows from \cite[Theorem 3.4]{U6} that $\gldim(\tails A)<\infty$.
By \cite[Proposition 4.4]{U6} and \cite[Theorem 3.8]{Mck},
indecomposable non-projective maximal Cohen-Macaulay $A$-modules
(up to isomorphism and degree shift of grading) are parameterized by points of
$\Proj A^! = \cV(xy+z^2,x^2-y^2) \subset \P^2$.
Using this, one can check that the graded $A$-module
$$ X:=A\oplus X_1\oplus X_2\oplus X_3\oplus X_4$$
is a representation generator of $\CM{A}$, that is, a $1$-cluster tilting module, where
\begin{align*}
&X_1:= A/(x-y+z)A, \quad X_2:= A/(x-y-z)A,\\
&X_3:= A/(x+y+iz)A, \quad X_4:= A/(x+y-iz)A
\end{align*}
and $i$ is a primitive 4th root of unity.
Clearly $X$ is basic, so it is $\nu$-stable.
Since generators of $X$ are concentrated in degree $0$, every graded $A$-module homomorphism $X \to X(-s)$ has to be zero for any positive integer $s \in \N^+$,
so we have $\uEnd(X)_{<0}=0$.
Hence we obtain that $B = \uEnd_A(X)$ is a two-sided noetherian AS-regular algebra over $B_0$ of dimension $2$ and of Gorenstein parameter $1$
satisfying $\tails B \cong \tails A$ by \corref{cor:main}.
We can calculate that the Hilbert series $H_B(t)$ is  $\frac{9(1+t)}{(1-t)^2}$.
Furthermore $B_0$ is isomorphic to the path algebra $kQ$ where
\[
Q=
\vcenter{
\xymatrix@C=0.8pc@R=0.02pc{
 \bullet \ar[rd]& &\bullet\ar[ld] \\
 &\bullet \\
 \bullet \ar[ru] & &\bullet \ar[lu]
}},
\]
so $\gldim B_0 < \infty$. Hence we obtain  %%%that $B$ is the preprojective algebra of the quiver $Q$, and 
\[\Db{\tails A} \cong \Db{\tails B} \cong \Db{\mod B_0} \cong \Db{\mod kQ} \]
by \cite[Theorem 4.14]{MM}.
As a remark, for any $2$-dimensional commutative Cohen-Macaulay algebra $C$ of finite Cohen-Macaulay representation type generated in degree $1$, it follows that
$\tails C \cong \tails k[x,y] \cong \coh \P^1$ by \cite[Theorem on page 347]{EH}, so
$\Db{\tails C} \cong \Db{\coh \P^1} \cong \Db{\mod kQ'}$
where
$
Q'=
\vcenter{
\xymatrix@C=2pc@R=1pc{
\bullet \ar@<0.7ex>[r] \ar@<-0.7ex>[r] &\bullet \\
}}.
$
Thus we see that $\Db{\tails A} \cong \Db{\tails B} \ncong \Db{\tails C}$.
\end{ex}

\section*{Acknowledgment} 
The author thanks the referee for his/her careful reading of the manuscript and helpful suggestions.

\end{document}